\documentclass[leqno,12pt]{amsart} 
\usepackage[top=30truemm,bottom=30truemm,left=25truemm,right=25truemm]{geometry}
\usepackage{amssymb}
\usepackage{amsmath}
\usepackage{amsthm}
\usepackage{amscd}
\usepackage{mathrsfs}
\usepackage{graphicx}
\usepackage[dvips]{color}
\usepackage{url}
  \makeatletter

\@addtoreset{equation}{section}
\@namedef{subjclassname@2020}{\textup{2020} Mathematics Subject Classification}
\makeatother
\theoremstyle{plain} 
\newtheorem{theorem}{\indent\bf Theorem}[section]
\newtheorem{lemma}[theorem]{\indent\bf Lemma}
\newtheorem{corollary}[theorem]{\indent\bf Corollary}
\newtheorem{proposition}[theorem]{\indent\bf Proposition}

\newtheorem{conjecture}[theorem]{\indent\bf Conjecture}
\theoremstyle{definition} 
\newtheorem{definition}[theorem]{\indent\bf Definition}
\newtheorem{remark}[theorem]{\indent\bf Remark}
\newtheorem{example}[theorem]{\indent\bf Example}

\newtheorem{question}[theorem]{\indent\bf Question}

\newcommand{\ddbar}{\partial \overline{\partial}}

\newcommand{\ai}{\sqrt{-1}}

\newcommand{\R}{\mathbb{R}}

\newcommand{\C}{\mathbb{C}}
\newcommand{\N}{\mathbb{N}}

\newcommand{\dz}{\mathrm{d}z}

\newcommand{\dbarz}{\mathrm{d}\bar{z}}

\newcommand{\dtau}{\mathrm{d}\tau}
\newcommand{\dbartau}{\mathrm{d}\bar{\tau}}

\newcommand{\Id}{\mathrm{Id}}
\newcommand{\Rea}{\mathrm{Re}}

\newcommand{\Grif}{\mathrm{Grif.}}
\newcommand{\Nak}{\mathrm{Nak.}}
\newcommand{\sasyugo}{\setminus\!}

\begin{document}
\pagestyle{plain}
\thispagestyle{plain}

\title[$L^2$-extension indices, sharper estimates and curvature positivity]
{$L^2$-extension indices, sharper estimates and curvature positivity}

\author[T. INAYAMA]{Takahiro INAYAMA}
\address{Department of Mathematics\\
Faculty of Science and Technology\\
Tokyo University of Science\\
2641 Yamazaki, Noda\\
Chiba, 278-8510\\
Japan
}
\email{inayama\_takahiro@rs.tus.ac.jp}
\email{inayama570@gmail.com}
\subjclass[2020]{32A36, 32U05}
\keywords{ 
Ohsawa--Takegoshi extension theorem, $L^2$-extension, plurisubharmonic function, Griffiths positivity, $L^2$-extension index.
}

\begin{abstract}
 In this paper, we introduce a new concept of $L^2$-extension indices.
 This index is a function that gives the minimum constant with respect to the $L^2$-estimate of an Ohsawa--Takegoshi-type extension at each point.
 By using this notion, we propose a new way to study the positivity of curvature. 
 We prove that there is an equivalence between how sharp the $L^2$-extension is and how positive the curvature is.
 New examples of sharper $L^2$-extensions are also systematically given. 
 As applications, we use the $L^2$-extension index to study Pr\'ekopa-type theorems and to study the positivity of a certain direct image sheaf. 
 We also provide new characterizations of pluriharmonicity and curvature flatness. 
\end{abstract}


\maketitle
\setcounter{tocdepth}{2}

\section{Introduction}\label{sec:intro}

The Ohsawa--Takegoshi $L^2$-extension theorem \cite{OT87} is a fundamental theorem concerning the $L^2$-extension of a holomorphic function.
This theorem has been applied to a variety of fields, not only in complex analysis, but also in algebraic geometry.
For example, by using this theorem, Siu established the invariance of plurigenera \cite{Siu98}.
The statement of the $L^2$-extension theorem is as follows.
Let $\Omega\subset \C^n$ be a bounded pseudoconvex domain with $\Omega\subset \C^{n-1}\times \{|z_n|<r\}$ for $r>0$, $\varphi$ be a plurisubharmonic (psh, for short) function on $\Omega$ and $H:=\Omega \cap \{ z_n=0\} (\neq \varnothing)$.
Then for every holomorphic function $f$ on $H$ with finite $L^2$-norm, there is a holomorphic function $F$ on $\Omega$ such that $F|_{H}=f$ and 
\begin{equation}\label{eq:introOT}
\int_\Omega |F|^2e^{-\varphi}\leq Cr^2 \int_H |f|^2e^{-\varphi}
\end{equation}
for some universal positive constant $C>0$, which is independent of the choice of $\varphi$ and $f$. 
Thanks to the celebrated results by B\l ocki \cite{Blo13} and Guan--Zhou \cite{GZ15}, we can take $C=\pi$, which is known as the optimal constant.

Conversely, it is known that if an upper-semicontinuous function $\varphi$ satisfies the above optimal $L^2$-extension theorem, $\varphi$ is necessarily psh. This phenomenon is called {\it the minimal extension property} or {\it the optimal $L^2$-extension property} (see Theorem \ref{thm:minimalextensionproperty}). 
Roughly speaking, we can say that the optimality of the $L^2$-extension is equivalent to the plurisubharmonicity of the weight $\varphi$. 
This idea has led to numerous useful applications. 
Indeed, there are Guan--Zhou's approach \cite{GZ15} to Berndtsson's log-plurisubharmonicity of the relative Bergman kernel \cite{Ber06}, new studies on positivity for vector bundles \cite{HPS18, DWZZ18, HI21, DNW21, DNWZ22, KP21} and a simple proof of Pr\'ekopa's theorem \cite{Ina22JGA}. 
For these reasons, this property can be considered as important. 

If we allow the constant $C$ in (\ref{eq:introOT}) to depend on the weight, we can sharpen this estimate. 
This study was initiated by Hosono \cite{Hos19}, and generalized by Kikuchi \cite{Kik22} and Xu--Zhou \cite{XZ22}.
In this paper, we call these types of estimates \textit{sharper estimates}. 
The sharper estimate asserts that the strict inequality holds in the estimate (\ref{eq:introOT}) under some geometric settings.
Hence, it is natural and important to ask the question, ``What do we know about the positivity of $\varphi$ from the assumption that $\varphi$ satisfies the condition of sharper estimates?"
To think about the problem a little more precisely, we introduce a notion of $L^2$-extension indices.

\begin{definition}\label{intro:l2index}
	Let $\Omega\subset \C$ be a bounded domain, $\varphi$ be a smooth function on $\Omega$ and $\delta(a):= \sup\{ r>0 \mid \Delta(a;r)=\{z\in \C \mid |z-a|<r \} \subset \Omega\}$. 
	We also let $\Omega_\delta$ be $\Omega_\delta=\{ (a,r)\in \Omega \times \R \mid 0<r<\delta(a) \}$.
	Then we define {\it the $L^2$-extension index} $L_\varphi$ of $\varphi$ on $\Omega_\delta$ by 
	\begin{align*}
	L_\varphi (a, r)&= \frac{1}{\pi r^2 K_{\Delta(a;r),\varphi}(a)},
    \end{align*}
where $K_{\Delta(a;r),\varphi}$ is the weighted Bergman kernel on $\Delta(a;r)$ with respect to the weight $\varphi$.
\end{definition}

Let $A^2(\Omega,\varphi)$ denote $A^2(\Omega,\varphi)=\{ f\in \mathcal{O}(\Omega) \mid \int_\Omega |f|^2e^{-\varphi}<+\infty \}$, where $\mathcal{O}(\Omega)$ is the space of holomorphic functions on a domain $\Omega$. 
Since 
$$
\frac{1}{\pi r^2 K_{\Delta(a;r),\varphi}(a)}
=\inf\left\{ \frac{\int_{\Delta(a;r)}|f|^2e^{-\varphi}}{\pi r^2 e^{-\varphi(a)}} ~\middle|~ f\in A^2(\Delta(a;r),\varphi), f(a)=1  \right\},
$$
the $L^2$-extension index $L_\varphi$ gives the best constant (depending on the weight $\varphi$) with respect to the $L^2$-estimate of the Ohsawa--Takegoshi extension at each point. 
Using this notion, we see that the question to be considered can be summarized as follows. 

\begin{question}\label{mainque:sharper}
	Keep the notation.
	If $L_\varphi(a,r)<1$ for each point $(a,r)\in \Omega_\delta$, what can we know about $\varphi$?
	For example, can we say that $\varphi$ is strictly psh?
\end{question}

In this paper, 
we use the $L^2$-extension index to give a quantitative estimate of the curvature of $\varphi$ and give an answer to Question \ref{mainque:sharper}. 
One of the main theorems in this paper is the following.

\begin{theorem}\label{mainthm:1jigen}
	Let $\Omega\subset \C$ be a bounded domain, $\varphi$ be a smooth function on $\Omega$ and $\omega=\ai \dz\wedge\dbarz$. 
		Suppose that for any point $a\in\Omega$, there exist $\gamma_a\in (0,\delta(a))$ and a semi-positive lower semi-continuous function $g_a\colon[0, \gamma_a]\to \R_{\geq 0}$ such that $L_\varphi(a,r)\leq e^{-g_a(r)r^2}$ for $r\in (0,\gamma_a)$. 
		Then 
		$$
		\ai \ddbar \varphi(a)\geq 2g_a(0)\omega.
		$$
\end{theorem}

If we take $g_a\equiv 0$, this theorem corresponds to the minimal extension property. 
As a very special case, it holds that $\ai\ddbar\varphi\geq 2c\omega$ if $L_\varphi(a,r)\leq e^{-cr^2}<1$ on $\Omega_\delta$ for $(a,r)\in \Omega_\delta$ and a positive constant $c> 0$, which is one answer to Question \ref{mainque:sharper}. 
We can also prove that this estimate is best possible in suitable sense (see Corollary \ref{cor:douchi} and \ref{cor:douchi2}).
As an application, we can discuss a quantification of Pr\'ekopa's theorem in Section \ref{sec:1jigen} (see Theorem \ref{thm:prekopateiryouka}). 



%

%
We can also establish a higher rank analogue of Theorem \ref{mainthm:1jigen} as follows.

\begin{theorem}\label{mainthm:koujigen}
	Let $\Omega$ be a bounded domain in $\C$, $E\to \Omega$ be a holomorphic vector bundle of rank $r$ and $h$ be a smooth Hermitian metric on $E$. 
	Assume that for any point $a\in \Omega$, there exist $\gamma_a\in(0, \delta(a))$ and a semi-positive lower semi-continuous function $g_a:[0, \gamma_a]\to \R_{\geq 0}$ such that $L_h(a,r, \xi)\leq e^{-g_a(r)r^2}$ for any $\xi\in E_a\sasyugo\!\{0\}$ and any $r\in(0,\gamma_a)$. 
	Then 
	$$
	\ai\Theta_h(a)\geq 2g_a(0)\omega\otimes \Id_E,
	$$
	where $\ai\Theta_h$ is the Chern curvature of $(E,h)$. 
\end{theorem}

For the definition of $L_h(a,r, \xi)$, see Definition \ref{def:l2indexvector}. 
As an application, we study the relationship between the positivity of direct image sheaves and the sharper estimate.
	Let $\pi:\mathcal{X}\to B$ be a proper holomorphic submersion, where $\mathcal{X}$ is a K\"ahler manifold of dimension $n+1$ and $B$ is a bounded domain in $\C$. 
We also let $\mathcal{L}\to \mathcal{X}$ be a line bundle over $\mathcal{X}$ with a smooth Hermitian metric $h$. 
Assume that $\dim H^0(X_\tau, K_{X_\tau}\otimes L_\tau)$ has the same dimension for $\tau\in B$, where $X_\tau=\pi^{-1}(\tau)$ and $L_\tau=\mathcal{L}|_{X_\tau}$. 
Set $E_\tau:= H^0(X_\tau, K_{X_\tau}\otimes L_\tau)$.
Then 
$$
E:= \bigcup_{\tau\in B}\{ \tau\}\times E_\tau
$$
admits a structure of a holomorphic vector bundle and can be identified with the direct image bundle $\pi_\star(K_{\mathcal{X}/B}\otimes \mathcal{L})$, where $K_{\mathcal{X}/B}$ is the relative canonical bundle. 
Under this isomorphism, $\pi_\star(K_{\mathcal{X}/B}\otimes \mathcal{L})$ has a canonical Hermitian metric $H$ induced by $h$ as follows. 
For $\tau\in B$ and $u\in E_\tau=H^0(X_\tau, K_{X_\tau}\otimes L_\tau)\cong H^{n,0}(X_\tau, L_\tau)$, 
$$
|u|^2_{H_\tau}:=\int_{X_\tau}c_n h_\tau u\wedge \bar{u}, 
$$
where $c_n=\sqrt{-1}^{n^2}$. 
Here we denote by $h_\tau$ the restriction of $h$ to $L_\tau$. 
In this setting, we have the following theorem.

\begin{theorem}\label{mainthm:directimage}
Keep the notation above. 
Assume that $H^0(X_\tau, K_{X_\tau}\otimes L_\tau)$ has the same dimension for each $\tau\in B$.
Let $g\colon B\to \R_{\geq 0}$ be a semi-positive function on $B$. 
	Then the following properties are equivalent$\colon$
	\begin{enumerate}
		\item the Chern curvature $\ai\Theta_H$ of the direct image bundle $\pi_\star(K_{\mathcal{X}/B}\otimes \mathcal{L})=E$ is positively curved such as $\ai\Theta_H\geq g\omega \otimes \Id_{\pi_\star(K_{\mathcal{X}/B}\otimes \mathcal{L})}$. 
		\item for any point $a\in B$, any $\xi\in E_a= H^0(X_a, K_{X_a}\otimes L_a)$ and any $\varepsilon>0$, there exists $r_\varepsilon\in (0, \delta(a))$ such that for any $r\in(0,r_\varepsilon)$, there exists a holomorphic section $s\in H^0(\pi^{-1}(\Delta(a;r)), K_\mathcal{X}\otimes \mathcal{L})$ such that $s|_{X_a}=\xi\wedge \dtau$ and 
		$$
		\int_{\pi^{-1}(\Delta(a;r))}c_{n+1}hs\wedge \bar{s} \leq e^{-\max\{ \frac{(g(a)-\varepsilon)}{2}, 0\} r^2}\pi r^2 \int_{X_a}c_n h_a\xi\wedge \bar{\xi}. 
		$$
	\end{enumerate}
\end{theorem}
The correspondence is considered to provide a new direction for the study of positivity of direct images.

Finally, we see that the $L^2$-extension index leads to a new and interesting characterization of (pluri)harmonic functions. 
On a one-dimensional domain, as is well known, a subharmonic function is characterized by the mean value inequality, and when the equality of the inequality holds, the function becomes a harmonic function. 
As discussed above, the (pluri)subharmonicity of the weight can be characterized by the inequality part of the optimal Ohsawa--Takegoshi $L^2$-extension theorem.
Therefore, based on the analogy with the above, the following natural question arises:

\begin{question}\label{que:harmonic}
Are harmonic functions characterized by the equality part of the optimal Ohsawa--Takegoshi $L^2$-extension theorem?
\end{question} 

We can give an affirmative answer to the above question under an appropriate formulation. 
Indeed, we get the following result. 

\begin{theorem}\label{mainthm:harmoniccharacterization}
	Let $\varphi$ be a smooth function on a domain $\Omega\subset \C$. 
	Then the following are equivalent$\colon$
	\begin{enumerate}
		\item $\ai \ddbar  \varphi=0$. 
		\item $L_\varphi= 1$, that is, for any $a\in \Omega$ and $r\in (0,\delta(a))$, there exists a \textbf{\textit{unique}} holomorphic function $f$ on $\Delta(a;r)$ satisfying $f(a)=1$ and 
		\[
		 \int _{\Delta(a;r)}|f|^2e^{-\varphi} \leq \pi r^2 e^{-\varphi(a)}. 
		\]
	\end{enumerate}
\end{theorem}

Pseudoconvexity is related to the richness of certain holomorphic functions (such as Stein manifolds).  
This theorem states that such a thing also holds with respect to subharmonicity. 
In fact, if there are many holomorphic functions that satisfy the optimal Ohsawa-Takegoshi condition, the metric becomes subharmonic, and if there is only one such holomorphic function, the metric becomes harmonic.
On an $n$-dimensional domain, we can also get such a characterization of pluriharmonicity (see Appendix \ref{sec:further}). 
As a higher rank analogue of Theorem \ref{mainthm:harmoniccharacterization}, we get a new characterization of curvature flatness as follows. 

\begin{theorem}\label{mainthm:flatness}
	Let $\Omega, E$ and $h$ be the same thing as in Theorem \ref{mainthm:koujigen}. 
	Then the following are equivalent$\colon$
	\begin{enumerate}
		\item $\ai \Theta_h =0$. 
		\item $L_h=1$, that is, for any $a\in \Omega, r\in (0,\delta(a))$ and $\xi\in E_a\sasyugo \{ 0\}$, there exists a \textbf{\textit{unique}} holomorphic section $s$ of $E$ on $\Delta(a;r)$ satisfying $s(a)=\xi$ and \[
			\int_{\Delta(a;r)} |s|^2_h \leq \pi r^2 |\xi|^2_{h(a)}. 
		\]
	\end{enumerate}
\end{theorem}

The organization of the manuscript is as follows.
In Section \ref{sec:prelim}, we recall positivity notions for vector bundles and basic facts about the optimal $L^2$-extension property.
In this section, we also introduce notions of $L^2$-extension indices for line bundles and vector bundles. 
In Section \ref{sec:1jigen}, we give a proof of Theorem \ref{mainthm:1jigen} and discuss Pr\'ekopa-type theorems.
In Section \ref{sec:koujigen}, we prove Theorem \ref{mainthm:koujigen} and \ref{mainthm:directimage}. 
In Section \ref{sec:harmonicity}, we establish new characterizations of harmonicity and curvature flatness, and prove Theorem \ref{mainthm:harmoniccharacterization} and \ref{mainthm:flatness}. 
At last, in Appendix \ref{sec:further}, we discuss a generalization of the $L^2$-extension index and related problems.

\vskip3mm
{\bf Acknowledgment. } 
The author would like to thank Genki Hosono for helpful comments. 
He also expresses gratitude to the reviewer for providing valuable advice to improve the paper.
He is supported by Japan Society for the Promotion of Science, Grant-in-Aid for Research Activity Start-up (Grant No. 21K20336) and Grant-in-Aid for Early-Career Scientists (Grant No. 23K12978).

\section{Preliminaries}\label{sec:prelim}
\subsection{Positivity notions for vector bundles}

In this subsection, let us recall basic positivity notions for vector bundles.
Let $X$ denote a complex manifold of $\dim n$, $E\to X$ be a holomorphic vector bundle of rank $r$ and $h$ be a smooth Hermitian metric on $E$. 
We denote by $\Theta_h$ the Chern curvature of $(E,h)$. 
Taking local coordinates $(z_1, \ldots, z_n)$ of $X$ and an orthonormal frame $(e_1, \ldots, e_r)$ of $E$ at some fixed point $x_0\in X$, we write 
$$
\ai\Theta_h=\sum_{\substack{1\leq j,k\leq n\\ 1\leq \lambda, \mu\leq r}}c_{jk\lambda\mu}\dz_j\wedge \dbarz_k\otimes e_\lambda^\star\otimes e_\mu.
$$
By using this expression, we can define the associated Hermitian form $\widetilde{\Theta}_h$ on $T_X\otimes E$ at $x_0$ by 
$$
\widetilde{\Theta}_h(\tau,\tau)=\sum_{\substack{1\leq j,k\leq n\\ 1\leq \lambda, \mu\leq r}}c_{jk\lambda\mu}\tau_{j\lambda}\bar{\tau}_{k\mu},
$$
for $\tau=\sum_{j,\lambda}\tau_{j\lambda}\frac{\partial}{\partial z_j}\otimes e_\lambda\in T_X\otimes E$, where $T_X$ is the holomorphic tangent bundle of $X$. 
Then $(E,h)$ is said to be {\it Nakano positive} (respectively, {\it Nakano semi-positive}) if $\widetilde{\Theta}_h(\tau,\tau)>0$ (respectively, $\widetilde{\Theta}_h(\tau,\tau)\geq 0$) for all non-zero elements $\tau\in T_X\otimes E$, and $(E,h)$ is said to be {\it Griffiths positive} (respectively, {\it Griffiths semi-positive}) if $\widetilde{\Theta}_h(v\otimes s,v\otimes s)>0$ (respectively, $\widetilde{\Theta}_h(v\otimes s,v\otimes s)\geq 0$) for all non-zero elements $v\in T_X$ and $s\in E$. Corresponding negativity is defined similarly. 
For two Hermitian forms $A, B$ on $T_X\otimes E$, we write $A\geq_{\Nak}B$ (respectively, $A\geq_{\Grif}B$) if $A(\tau,\tau)\geq B(\tau,\tau)$ (respectively, $A(v\otimes s, v\otimes s)\geq B(v\otimes s, v\otimes s)$) for all $\tau\in T_X\otimes E$ (respectively, $v\in T_X$ and $s\in E$). 
For two Hermitian metrics $h_A, h_B$ on $E$, we write $\ai\Theta_{h_A}\geq_{\Nak}\ai\Theta_{h_B}$ (respectively, $\ai\Theta_{h_A}\geq_{\Grif}\ai\Theta_{h_B}$) when the corresponding Hermitian forms $\widetilde{\Theta}_{h_A}$ and $\widetilde{\Theta}_{h_B}$ satisfy $\widetilde{\Theta}_{h_A}\geq_{\Nak}\widetilde{\Theta}_{h_B}$ (respectively, $\widetilde{\Theta}_{h_A}\geq_{\Grif}\widetilde{\Theta}_{h_B}$).

By definition, we clearly see that Nakano positivity is a stronger positivity notion that Griffiths positivity.
Note that, if $\dim X=1$ or rank $E=1$, these notions coincide. 
When $\dim X=1$, we just use the terminology ``(semi-)positively curved" and write $\ai\Theta_h\geq 0$, $\ai\Theta_{h_A}\geq \ai\Theta_{h_B}$ and so on. 
Throughout of this paper, we mainly focus on Griffiths (semi-)positivity and assume that $\Omega $ is a 1-dimensional domain in $\C$.

\subsection{Optimal $L^2$-extension property}

We explain the property called {\it the minimal extension property} \cite{HPS18} or {\it the optimal $L^2$-extension property} \cite{DNW21} \cite{DNWZ22} for a $1$-dimensional domain. 
In this article, we simply use the term ``the optimal $L^2$-extension property" since we follow the formulation of Deng--Ning--Wang--Zhou.

\begin{definition}[{minimal extension property \cite{HPS18}, optimal $L^2$-extension property \cite{DNW21}, \cite{DNWZ22}}]\label{def:minimalextensionproperty}
	Let $\varphi$ be an upper semi-continuous function on a domain $\Omega \subset \C$. 
	We say that $\varphi$ satisfies {\it the optimal $L^2$-extension property} if the following condition is satisfied: for each $a\in \Omega$ with $\varphi(a)\neq -\infty$ and any $r>0$ with $\Delta(a;r)\subset \Omega$, there exists a holomorphic function $f$ on $\Delta(a;r)$ satisfying $f(a)=1$ and 
	$$
	\int_{\Delta(a;r)}|f|^2e^{-\varphi}\leq \pi r^2 e^{-\varphi(a)}.
	$$
\end{definition}

In other words, if we can get the optimal $L^2$-extension from any point $a\in \Omega$ to $\Delta(a;r)$ with respect to the $L^2_\varphi$-norm, we say that $\varphi$ satisfies the optimal $L^2$-extension property. 
As mentioned in Introduction, we know that if $\varphi$ is psh, $\varphi$ satisfies the optimal $L^2$-extension property thanks to the work of B\l ocki \cite{Blo13} and Guan--Zhou \cite{GZ15}. 
Conversely, it is known that the following holds. 

\begin{theorem}[{\cite[Theorem 1.6]{DNW21}}, cf. \cite{GZ15}, \cite{HPS18}, \cite{DNWZ22}]\label{thm:minimalextensionproperty}
	Let $\varphi$ be an upper semi-continuous function on a domain $\Omega\subset \C$.  
	If $\varphi$ satisfies the optimal $L^2$-extension property, $\varphi$ is subharmonic. 
\end{theorem}

One main topic in this article is to generalize this result to the quantitative estimate of curvature positivity.

\subsection{$L^2$-extension index for line bundles}\label{prelim:subsec:1jigen}
In this subsection, we list the basic properties about the $L^2$-extension index.
Throughout this subsection, we let $\Omega\subset\C$ be a bounded domain and $\varphi$ be a smooth function on $\Omega$. 
First, we show the following elementary but important property. 

\begin{lemma}\label{lem:attain}
	For any $(a,r)\in \Omega_\delta$, there exists a holomorphic function $f$ on $\Delta(a;r)$ satisfying $f(a)=1$ and 
	$$
	L_\varphi(a,r)= \frac{\int_{\Delta(a;r)}|f|^2e^{-\varphi}}{\pi r^2 e^{-\varphi(a)}}.
	$$
\end{lemma}

This lemma immediately follows due to the standard property of the weighted Bergman kernel. 
By using the notion of the $L^2$-extension index and Lemma \ref{lem:attain}, we can summarize Theorem \ref{thm:minimalextensionproperty} as follows. 
We use the same notation as in Theorem \ref{thm:minimalextensionproperty}.

\begin{theorem}\label{thm:iikae}
	If $L_\varphi\leq 1$, $\varphi$ is subharmonic. 
\end{theorem}

For another example, we can state \cite[Theorem 1.7]{DNW21} as follows. 

\begin{theorem}
	If $\log L_{m\varphi}/m \to 0$ pointwise as $m\to +\infty$, $\varphi$ is subharmonic. 
\end{theorem}

This type of theorem was proved in a more general setting \cite{DWZZ18}.
For the reader's convenience and to familiarize her or himself with the concept of the $L^2$-extension index, we give a quick proof of the theorem. 
The proof is almost the same as in \cite{DNW21}.

\begin{proof}
	Fix $(a,r)\in \Omega_\delta$ and $m\in \N$. 
	Thanks to Lemma \ref{lem:attain}, we can obtain a holomorphic function $f_m$ on $\Delta(a;r)$ satisfying $f_m(a)=1$ and 
	$$
	L_{m\varphi}(a,r)=\frac{\int_{\Delta(a;r)}|f_m|^2e^{-m\varphi}}{\pi r^2 e^{-m\varphi(a)}}, 
	$$
	that is, 
	$$
	L_{m\varphi}(a,r)e^{-m\varphi(a)}= \int_{\Delta(a;r)}|f_m|^2e^{-m\varphi}\frac{1}{\pi r^2}. 
	$$
	Taking the logarithm of the above inequality and using Jensen's inequality, we get 
	\begin{align*}
		-m\varphi(a)+\log L_{m\varphi}(a,r)&\geq \log \left( \int_{\Delta(a;r)}|f_m|^2e^{-m\varphi}\frac{1}{\pi r^2}\right)\\
		&\geq \int_{\Delta(a;r)} (\log |f_m|^2-m\varphi)\frac{1}{\pi r^2}\\
		&\geq \frac{1}{\pi r^2}\int_{\Delta(a;r)} -m\varphi
	\end{align*}
since $\log |f_m|^2$ is psh and $f_m(a)=1$. 
Hence, it follows that 
$$
\frac{1}{\pi r^2}\int_{\Delta(a;r)} \varphi \geq \varphi(a)- \frac{\log L_{m\varphi}(a,r)}{m}. 
$$
Taking $m\to +\infty$, we get 
$$
\frac{1}{\pi r^2}\int_{\Delta(a;r)} \varphi \geq \varphi(a),
$$
which completes the proof. 
\end{proof}

We can also show the following property. 
This result describes the behavior of $L_\varphi(a, \cdotp)$ near the origin.

\begin{proposition}\label{prop:lim0}
	$\lim_{r\to +0}L_\varphi(a, r)=1$ for any $a\in \Omega$. 
\end{proposition}

\begin{proof}
	Fix $a\in \Omega$. 
	Let $\varepsilon>0$ be an arbitrary positive number. 
	For each $r\in (0,\delta(a))$, due to Lemma \ref{lem:attain}, there is a holomorphic function $f_r$ on $\Delta(a;r)$ satisfying $f_r(a)=1$ and 
	$$
	L_\varphi(a,r)=\frac{\int_{\Delta(a;r)}|f_r|^2e^{-\varphi}}{\pi r^2e^{-\varphi(a)}}. 
	$$
	Since $\varphi$ is continuous, there exists $r_1\in (0,\delta(a))$ such that for any $z\in \Delta(a;r_1)$, $|\varphi(z)-\varphi(a)|<\varepsilon$. 
	Using the mean-value inequality for $|f_r|^2$, we have that 
	$$
	\int_{\Delta(a;r)}|f_r|^2e^{-\varphi}\geq e^{-\varphi(a)-\varepsilon}\int_{\Delta(a;r)}|f_r|^2\geq \pi r^2e^{-\varphi(a)-\varepsilon}.
	$$
	Hence, it follows that 
	$$
	L_\varphi(a,r)=\frac{\int_{\Delta(a;r)}|f_r|^2e^{-\varphi}}{\pi r^2e^{-\varphi(a)}}\geq e^{-\varepsilon}> 1-\varepsilon
	$$
	for any $r\in (0,r_1)$. 
	
	On the other hand, we fix $R\in (0,\delta(a))$. 
	Thanks to the mean-value theorem, for any $r\in (0,R)$, there exists $\zeta_r\in\Delta(a;r)$ such that 
	$$
	\int_{\Delta(a;r)}|f_R|_{\Delta(a;r)}|^2e^{-\varphi}=\pi r^2 |f_R(\zeta_r)|^2e^{-\varphi(\zeta_r)}.
	$$
	Then we get 
	$$
	\frac{\int_{\Delta(a;r)}|f_R|_{\Delta(a;r)}|^2e^{-\varphi}}{\pi r^2e^{-\varphi(a)}}=\frac{|f_R(\zeta_r)|^2e^{-\varphi(\zeta_r)}}{e^{-\varphi(a)}}.
	$$
	Since $\zeta_r\to a$ as $r\to+0$, it holds that 
	$$
	\lim_{r\to +0}\frac{\int_{\Delta(a;r)}|f_R|_{\Delta(a;r)}|^2e^{-\varphi}}{\pi r^2e^{-\varphi(a)}}=1.
	$$
	Then there is $r_2\in (0,R)$ such that for any $r\in(0, r_2)$ 
	$$
	1-\varepsilon<\frac{\int_{\Delta(a;r)}|f_R|_{\Delta(a;r)}|^2e^{-\varphi}}{\pi r^2e^{-\varphi(a)}}<1+\varepsilon. 
	$$
	By the definition of $L_\varphi$, we see that 
	$$
	L_\varphi(a,r)\leq \frac{\int_{\Delta(a;r)}|f_R|_{\Delta(a;r)}|^2e^{-\varphi}}{\pi r^2e^{-\varphi(a)}}< 1+\varepsilon
	$$
	for $r\in(0,r_2)$. 
	Thus, letting $r_0:=\min\{ r_1, r_2\}$, we have that 
	$$
	1-\varepsilon<L_\varphi(a,r)<1+\varepsilon
	$$
	for any $r\in (0,r_0)$. 
\end{proof}

At the last of this subsection, we show a specific calculation of the $L^2$-extension index. 

\begin{example}\label{ex:optimalexample}
	Let $\Delta:=\Delta(0;1)$ and $\varphi(z)=\lambda|z|^2$ for $\lambda>0$. We can easily see that 
	$$
	L_{\lambda |z|^2}(0,r)=\frac{\int_{\Delta(0;r)}e^{-\lambda|z|^2}}{\pi r^2e^{-\varphi(0)}}
	$$
	for $r\in (0,1)$. Then it follows that 
	$$
	L_{\lambda|z|^2}(0,r)=\frac{1}{\lambda r^2}(1-e^{-\lambda r^2})
	$$
	due to the simple computation. 
\end{example}

\subsection{$L^2$-extension index for vector bundles}\label{prelim:subsec:koujigen}

In this subsection, let us introduce the higher rank analogues of the results in Section \ref{prelim:subsec:1jigen}.
First, we define $L^2$-extension indices for smooth Hermitian vector bundles.
Throughout this subsection, we also let $\Omega$ be a bounded domain in $\C$.

\begin{definition}\label{def:l2indexvector}
	Let $\pi:E\to \Omega$ be a holomorphic vector bundle over $\Omega$ and $h$ be a smooth Hermitian metric on $E$.
	We define {\it the $L^2$-extension index} $L_h$ of $h$ on $\Omega_\delta\times_\Omega E\sasyugo \{0\}$, by
	$$
	L_h(a,r, \xi)=\inf\left\{ \frac{\int_{\Delta(a;r)}|s|^2_h}{\pi r^2 |\xi|^2_{h(a)}} ~\middle|~ s\in A^2(\Delta(a;r), h), s(a)=\xi \right\},
	$$ 
	where $E\sasyugo\{0\}$ is the complement of the zero section of $E$, $\Omega_\delta\times_\Omega E\sasyugo\{0\} =\{ (a, r, \xi)\in \Omega_\delta\times E\sasyugo \{0\} \mid a=\pi(\xi) \}$ and $A^2(\Delta(a;r), h)$ is the space of square integrable holomorphic sections of $E$ with respect to $h$ on $\Delta(a;r)$. 
\end{definition}

If rank$E=1$, the above definition is equal to Definition \ref{intro:l2index}. 
The same properties as in Subsection \ref{prelim:subsec:1jigen} also hold. 

\begin{lemma}\label{lem:vectorattain}
	For any $(a, r, \xi)\in \Omega_\delta\times_\Omega E\sasyugo\{0\}$, there is a holomorphic section $s\in A^2(\Delta(a;r), h)$ such that $s(a)=\xi$ and 
	$$
	L_h(a,r, \xi)=\frac{\int_{\Delta(a;r)}|s|^2_h}{\pi r^2 |\xi|^2_{h(a)}}. 
	$$
\end{lemma}

\begin{proof}
	Let $(a, r, \xi)\in \Omega_\delta\times_\Omega E\sasyugo\{0\}$ and $A=L_h(a,r, \xi)$. 
	For each $n\in\N$, there exists $s_n\in A^2(\Delta(a;r),h)$ such that $s_n(a)=\xi$ and
	$$
	A\leq \frac{\int_{\Delta(a;r)}|s_n|^2_h}{\pi r^2 |\xi|^2_{h(a)}}\leq A+\frac{1}{n}.
	$$
	
	We choose a frame $(e_1, \ldots, e_r)$ of $E$ and write $u=u^1e_1+\cdots +u^re_r$ for a holomorphic section $u$ of $E$. 
	Since $\Delta(a;r)$ is relatively compact in $\Omega$, there exist constants $c, C>0$ such that $c|u|^2\leq |u|^2_h\leq C|u|^2$ on $\Delta(a;r)$ for $u\in A^2(\Delta(a;r),h)$. 
	Here $|u|^2$ denotes $|u|^2=|u^1|^2+\cdots+|u^n|^2$. 
	Therefore, 
	$$
	\int_{\Delta(a;r)}|s_n^i|^2\leq \frac{1}{c}\pi r^2 |\xi|^2_{h(a)}\left(A+1\right)
	$$
	for all $n\in\N$ and $1\leq i\leq r$. 
	By Montel's theorem, we obtain a subsequence $\{ s^i_{n_k}\}_k$ of $\{s^i_n\}$ such that $\{ s^i_{n_k}\}_k$ converges compactly to some holomorphic function $s^i$ for each $i$.
	Extracting a subsequence and renumbering the subscripts, we may assume that for each $i$, $\{ s^i_n\}_n$ satisfies the above conditions. 
	We define $s:=s^1e_1+\cdots+s^re_r\in H^0(\Delta(a;r),E)$.
	Note that $s(a)=\xi$.
	Then Fatou's lemma implies that 
	\[
		\frac{\int_{\Delta(a;r)}|s|^2_h}{\pi r^2 |\xi|^2_{h(a)}}\leq A,
	\]
which completes the proof. 
\end{proof}

\begin{proposition}\label{prop:lim0vector}
	For any $\xi\in E\sasyugo \Omega$,
	$\lim_{r\to +0}L_h(a,r, \xi)=1$, where $a=\pi(\xi)$. 
\end{proposition}

\begin{proof}
	Let $\varepsilon>0$. We choose an orthonormal frame $(e_1, \ldots, e_r)$ of $E$ at $a\in\Omega$. 
	Let us consider the following value 
	$$
	\frac{|s|^2_{h}}{|s|^2}
	$$
	for $s=s^1e_1+\cdots+s^re_r\in H^0(\Delta(a;r),E)$, where $|s|^2=|s^1|^2+\cdots+|s^r|^2$. 
	Note that its minimum value is the smallest eigenvalue of $h$, which varies continuously.
	Since $|s(a)|^2_h/|s(a)|^2=1$, there exists $r_1\in(0,\delta(a))$, which is independent of $s$, such that for any $z\in\Delta(a;r_1)$
	$$
	\frac{|s|^2_{h}}{|s|^2}>1-\varepsilon.
	$$
	We take $s_r\in H^0(\Delta(a;r),E)$ satisfying $s_r(a)=\xi$ and 
	$$
	L_h(a,r,\xi)=\frac{\int_{\Delta(a;r)}|s_r|^2_h}{\pi r^2 |\xi|^2_{h(a)}}
	$$
	for each $r\in(0,r_1)$. 
	Then 
	\begin{align*}
	L_h(a,r,\xi)&= \frac{\int_{\Delta(a;r)}|s_r|^2_h}{\pi r^2 |\xi|^2_{h(a)}}\\
	&> \frac{(1-\varepsilon)\int_{\Delta(a;r)}|s_r|^2}{\pi r^2 |\xi|^2_{h(a)}}\\
	&= \frac{(1-\varepsilon)\int_{\Delta(a;r)}(|s_r^1|^2+\cdots+|s_r^r|^2)}{\pi r^2(|\xi^1|^2+\cdots+|\xi^r|^2)}\\
	&\geq (1-\varepsilon). 
	\end{align*}
	
	On the other hand, for some fixed $R\in(0, \delta(a))$, we take $s_R$, which is the same as above. 
	By using the mean-value theorem, for any $r\in(0,R)$ there exists $\zeta_r\in\Delta(a;r)$ such that 
	$$
	\frac{\int_{\Delta(a;r)}|s_R|_{\Delta(a;r)}|^2_h}{\pi r^2|\xi|^2_{h(a)}}=\frac{|s_R(\zeta_r)|^2_{h(\zeta_r)}}{|\xi|^2_{h(a)}}.
	$$
	Since $\zeta_r\to a$ as $r\to +0$, we see that 
	$$
	\lim_{r\to +0}\frac{\int_{\Delta(a;r)}|s_R|_{\Delta(a;r)}|^2_h}{\pi r^2|\xi|^2_{h(a)}}=1. 
	$$
	
	By combining the two facts above and repeating the argument in the proof of Proposition \ref{prop:lim0}, we arrive at the conclusion that
	$
	\lim_{r\to +0}L_h(a,r, \xi)=1.
	$
\end{proof}

\section{$L^2$-extension indices for line bundles and curvature positivity}\label{sec:1jigen}

Our main goal in this section is to give a proof of Theorem \ref{mainthm:1jigen}. 
We also discuss Pr\'ekopa-type theorems.  
We show further applications as well.

First, in order to prove Theorem \ref{mainthm:1jigen}, we prepare the following elementary lemma. 

\begin{lemma}\label{lem:taylortenkai}
	Let $f(x)=\sqrt{\frac{e^x-1}{x}}$ for $x>0$.
	For any positive number $\delta>0$, there exists $\eta_\delta>0$ such that 
	$$
	\sqrt{\frac{e^x-1}{x}}\leq e^{\left(\frac{1}{4}+\frac{\delta}{2}\right)x}
	$$ 
	on $(0, \eta_\delta)$. 
\end{lemma}

\begin{proof}
	It is enough to show that 
	$$
	\frac{e^x-1}{x}\leq e^{\left(\frac{1}{2}+\delta\right)x}.
	$$
	The Taylor expansions of $\frac{e^x-1}{x}$ and $e^{\left(\frac{1}{2}+\delta\right)x}$ are given by 
	\begin{align*}
	\frac{e^x-1}{x}&=1+\frac{x}{2}+\cdots + \frac{x^n}{(n+1)!}+\cdots,\\
	e^{\left(\frac{1}{2}+\delta\right)x}&=1+\left(\frac{1}{2}+\delta \right)x+\cdots +\frac{\left(\frac{1}{2}+\delta\right)^nx^n}{n!}+\cdots .
    \end{align*}
Hence, there exists $\eta_\delta>0$ such that 
	$$
\sqrt{\frac{e^x-1}{x}}\leq e^{\left(\frac{1}{4}+\frac{\delta}{2}\right)x}
$$ 
on $(0, \eta_\delta)$ since 
$
e^{\left(\frac{1}{2}+\delta\right)x}-\frac{e^x-1}{x}=\delta x + o(x). 
$
\end{proof}

Then we prove Theorem \ref{mainthm:1jigen}. 

\begin{proof}[Proof of Theorem \ref{mainthm:1jigen}]
	Fix an arbitrary point $a\in\Omega$.
	We also take $\gamma_a$ and a lower semi-continuous function $g_a\colon[0, \gamma_a]\to \R_{\geq 0}$ as in Theorem \ref{mainthm:1jigen}. 
	If $g_a(0)>0$, we take an arbitrary number $\varepsilon$ such that $\varepsilon\in (0, g_a(0))$. 
	Since $g_a$ is lower semi-continuous, there exists $r_\varepsilon\in (0,\gamma_a)$ such that $g_a(0)-\varepsilon<g_a(r)$ for $r\in (0,r_\varepsilon)$. 
	By Lemma \ref{lem:attain}, for $r\in(0,r_\varepsilon)$ there exists a holomorphic function $f_r$ on $\Delta(a;r)$ satisfying $f_r(a)=1$ and 
	$$
	\int_{\Delta(a;r)}|f_r|^2e^{-\varphi}=L_\varphi(a,r)\pi r^2e^{-\varphi(a)}. 
	$$
	
	We now define a function $\varphi_{a,\varepsilon}$ by $\varphi_{a,\varepsilon}(z):=\frac{\varphi(z)}{2}-(g_a(0)-\varepsilon)|z-a|^2$.  
	Note that $\varphi_{a,\varepsilon}(a)=\frac{\varphi(a)}{2}$. 
	Then for $r\in(0, r_\varepsilon)$ we have that 
\begin{align*}
	\int_{\Delta(a;r)} |f_r|e^{-\varphi_{a,\varepsilon}}&=\int_{\Delta(a;r)}|f_r|e^{-\frac{\varphi}{2}}e^{(g_a(0)-\varepsilon)|z-a|^2}\\
	&\leq \left(\int_{\Delta(a;r)}|f_r|^2e^{-\varphi}\right)^{1/2}\left( \int_{\Delta(a;r)}e^{2(g_a(0)-\varepsilon)|z-a|^2}\right)^{1/2}\\
	&= \sqrt{L_\varphi(a,r)}\sqrt{\pi}re^{-\frac{\varphi(a)}{2}}\sqrt{\frac{\pi}{2(g_a(0)-\varepsilon)}\left( e^{2(g_a(0)-\varepsilon)r^2}-1\right)}\\
	&\leq \pi r^2 e^{-\frac{\varphi(a)}{2}}e^{-\frac{g_a(r)}{2} r^2}\sqrt{\frac{e^{2(g_a(0)-\varepsilon)r^2}-1}{2(g_a(0)-\varepsilon)r^2}}.
\end{align*}
from the assumption that $L_\varphi(a,r)\leq e^{-g_a(r)r^2}$.
We fix an arbitrary small number $\delta>0$. 
Thanks to Lemma \ref{lem:taylortenkai}, there exists $\eta_\delta>0$ such that 
$$
\sqrt{\frac{e^x-1}{x}}\leq e^{\left(\frac{1}{4}+\frac{\delta}{2}\right)x}
$$ 
on $(0, \eta_\delta)$. 
Set $r_\delta:=\sqrt{\frac{\eta_\delta}{2g_a(0)}}$, which is independent of $\varepsilon$. 
If $r\in (0,r_\delta)$, we have that 
$$
\sqrt{\frac{e^{2(g_a(0)-\varepsilon)r^2}-1}{2(g_a(0)-\varepsilon)r^2}}\leq e^{\left( \frac{1}{2}+\delta\right)\left(g_a(0)-\varepsilon\right)r^2}.
$$
Letting $r_{\varepsilon,\delta}:=\min\{ r_\varepsilon,r_\delta\}>0$, for $r\in(0,r_{\varepsilon,\delta})$ we obtain that 
\begin{align}
		\int_{\Delta(a;r)} |f_r|e^{-\varphi_{a,\varepsilon}}&=\int_{\Delta(a;r)}|f_r|e^{-\frac{\varphi}{2}}e^{(g_a(0)-\varepsilon)|z-a|^2}\label{ineq:1}\\
		&\leq \pi r^2 e^{-\frac{\varphi(a)}{2}}e^{-\frac{g_a(r)}{2} r^2}\sqrt{\frac{e^{2(g_a(0)-\varepsilon)r^2}-1}{2(g_a(0)-\varepsilon)r^2}}\label{ineq:2}\\
		&\leq \pi r^2 e^{-\frac{\varphi(a)}{2}}e^{-\frac{g_a(r)}{2} r^2}e^{\left(\frac{g_a(0)}{2}-\frac{\varepsilon}{2}\right)r^2}e^{(g_a(0)-\varepsilon)\delta r^2}\label{ineq:3}\\
		&\leq \pi r^2 e^{-\frac{\varphi(a)}{2}}e^{g_a(0)\delta r^2}\label{ineq:4}\\
		&= \pi r^2 e^{-\varphi_{a,\varepsilon}(a)}e^{g_a(0)\delta r^2}\label{ineq:5}
\end{align}
from $g_a(0)-\varepsilon-g_a(r)<0$. 
Taking the logarithm of the above inequality and using Jensen's inequality, we have 
$$
\frac{1}{\pi r^2}\int_{\Delta(a;r)}\varphi_{a,\varepsilon}\geq \varphi_{a,\varepsilon}(a)-g_a(0)\delta r^2. 
$$
We write the Taylor expansion of $\varphi_{a,\varepsilon}$ at $a\in\Omega$ by 
\begin{align*}
\varphi_{a,\varepsilon}(a+\zeta)&=\varphi_{a,\varepsilon}(a)+\frac{\partial\varphi_{a,\varepsilon}}{\partial z}(a)\zeta+\frac{\partial \varphi_{a,\varepsilon}}{\partial \bar{z}}(a)\bar{\zeta}\\
&+\frac{1}{2}\frac{\partial^2\varphi_{a,\varepsilon}}{\partial z^2}(a)\zeta^2+\frac{1}{2}\frac{\partial^2\varphi_{a,\varepsilon}}{\partial \bar{z}^2}(a)\bar{\zeta}^2+\frac{\partial^2\varphi_{a,\varepsilon}}{\partial z\partial \bar{z}}(a)|\zeta|^2+o (|\zeta|^2).
\end{align*}
Letting $\zeta=Re^{\ai\theta}$, we get 
$$
\frac{2}{\pi r^4}\int_{\Delta(0;r)}(\varphi_{a,\varepsilon}(a+\zeta)-\varphi_{a,\varepsilon}(a))=\frac{\partial^2 \varphi_{a,\varepsilon}}{\partial z\partial\bar{z}}(a)+o(1). 
$$
It also holds that 
$$
\frac{2}{\pi r^4}\int_{\Delta(0;r)}(\varphi_{a,\varepsilon}(a+\zeta)-\varphi_{a,\varepsilon}(a))\geq -2g_a(0)\delta.
$$
Taking $r\to +0$, we have that $\partial^2 \varphi_{a,\varepsilon}/\partial z\partial\bar{z}(a)\geq -2g_a(0)\delta$.
Since $\delta>0$ is arbitrary, taking the limit $\delta\to+0$, we get  $\partial^2 \varphi_{a,\varepsilon}/\partial z\partial\bar{z}(a)\geq 0$,
which implies that 
$$
\ai\ddbar\varphi(a)\geq 2(g_{a}(0)-\varepsilon)\omega. 
$$
Since $\varepsilon\in (0,g_a(0))$ is arbitrary, it follows that  
$$
\ai\ddbar\varphi(a)\geq 2g_{a}(0)\omega. 
$$

If $g_a(0)=0$, 
$\ai\ddbar\varphi(a)\geq 2g_a(0)\omega=0$ follows simply by Theorem \ref{thm:minimalextensionproperty} since $g_a\geq 0$. 

In both cases, 
$$
\ai\ddbar\varphi(a)\geq 2g_a(0)\omega
$$ 
is proved. 

\end{proof}

We remark that the estimate in Theorem \ref{mainthm:1jigen} is optimal in the following sense. 

\begin{corollary}\label{cor:douchi}
	Let $\varphi$ be a smooth function and $g$ be a non-negative function on a bounded domain $\Omega\subset \C$. 
	Fix $a\in \Omega$. 
	Then the following properties are equivalent$\colon$
	\begin{enumerate}
		\item For any $\varepsilon>0$, there exists $r_\varepsilon>0$ such that  
		$$
		L_\varphi(a,r)\leq e^{-\max\{ \frac{(g(a)-\varepsilon)}{2}, 0\}r^2}
		$$
		for $r\in(0,r_\varepsilon)$. 
		\item $\ai\ddbar\varphi(a)\geq g(a)\omega$. 
	\end{enumerate} 
\end{corollary}

Due to this corollary, we cannot expect that in Theorem \ref{mainthm:1jigen} there is a positive constant $\delta>0$ satisfying $\ai\ddbar\varphi(a)\geq (2+\delta)g_a(0)\omega$. 
If $g$ is a positive function, we can rephrase this corollary as follows. 
\begin{corollary}
Keep the setting. If $g$ is positive, the following are equivalent$\colon$
	\begin{enumerate}\label{cor:douchi2}
		\item For any $\varepsilon\in (0,g(a))$, there exists $r_\varepsilon>0$ such that 
		$$
		L_\varphi(a,r)\leq e^{-\frac{(g(a)-\varepsilon)}{2}r^2}
		$$
		for $r\in(0,r_\varepsilon)$. 
		\item $\ai\ddbar\varphi(a)\geq g(a)\omega$. 
	\end{enumerate}
\end{corollary}


\begin{proof}[Proof of Corollary \ref{cor:douchi}]
	$(1)\Longrightarrow (2).$ 
	We only need to show the proof in the case that $g(a)>0$.
	If $g(a)>0$, by taking $\varepsilon\in(0,g(a))$, we know that there exists $r_\varepsilon$ such that
	$$
	L_\varphi(a,r)\leq e^{-\frac{(g(a)-\varepsilon)}{2}r^2}
	$$ 
	for $r\in(0,r_\varepsilon)$. 
	Then applying Theorem \ref{mainthm:1jigen} and $\varepsilon\to +0$, we get 
	$$
	\ai\ddbar\varphi(a)\geq (g(a)-\varepsilon)\omega\to g(a)\omega.
	$$
	
	$(2)\Longrightarrow(1).$ 
	If $\varepsilon\geq g(a)$, we need to show $L_\varphi(a,r)\leq 1$. 
	The non-negativity of $g$ implies the plurisubharmonicity of $\varphi$.
	Using the optimal $L^2$-extension theorem \cite{Blo13}, \cite{GZ15}, we can say that $L_\varphi(a,r)\leq 1$ for $r\in (0,\delta(a))$. 
	Then it is enough to show the proof in the situation that $g(a)>0$ and $\varepsilon\in(0,g(a))$. 
	Set $c:= \frac{\partial\varphi}{\partial z}(a)$ and 
	$$
	f:=e^{c(z-a)}.
	$$
	Note that $f$ is a holomorphic function with $f(a)=1$. 
	We also define 
	$\Psi$ by
	$$
	\Psi=-|f|^2e^{-{\varphi}}=-e^{-({\varphi}-c(z-a)-\overline{c(z-a)})}. 
	$$
	Then we can compute 
	\begin{align*}
		\frac{\partial^2\Psi}{\partial z\partial \bar{z}}=&-e^{-({\varphi}-c(z-a)-\overline{c(z-a)})}\left( \frac{\partial{\varphi}}{\partial z}-c\right) \left( \frac{\partial{\varphi}}{\partial \bar{z}}-\overline{c}\right)\\
		&+e^{-({\varphi}-c(z-a)-\overline{c(z-a)})}\frac{\partial^2{\varphi}}{\partial z\partial \bar{z}}.
	\end{align*}
We get 
\begin{align*}
\frac{\partial^2\Psi}{\partial z\partial \bar{z}}(a)&=e^{-\varphi(a)}\frac{\partial^2{\varphi}}{\partial z\partial\bar{z}}(a)\\
&\geq e^{-\varphi(a)}g(a)>0. 
\end{align*}
We denote by $\widetilde{\Psi}$ 
$$
\widetilde{\Psi}:={\Psi}-e^{-\varphi(a)}(g(a)-\varepsilon)|z-a|^2.
$$
It holds that 
\begin{align*}
	\frac{\partial^2\widetilde{\Psi}}{\partial z\partial\bar{z}}(a)&=	\frac{\partial^2{\Psi}}{\partial z\partial\bar{z}}(a)-e^{-\varphi(a)}(g(a)-\varepsilon)\\
	&\geq \varepsilon e^{-\varphi(a)}>0.
\end{align*}
	Hence, there exists $r_\varepsilon$ such that $\widetilde{\Psi}$ is psh on $\Delta(a;r_\varepsilon)$. 
	Then we have that 
	$$
	\frac{1}{\pi r^2}\int_{\Delta(a;r)}\widetilde{\Psi}\geq \widetilde{\Psi}(a)=-e^{-\varphi(a)},
	$$
	that is, 
	$$
	e^{-\varphi(a)}-e^{-\varphi(a)}(g(a)-\varepsilon)\frac{1}{\pi r^2}\int_{\Delta(a;r)}|z-a|^2\geq \frac{1}{\pi r^2}\int_{\Delta(a;r)}|f|^2e^{-\varphi}
	$$
	for $r\in(0, r_\varepsilon)$. 
	In summary, we see that 
	$$
	e^{-\varphi(a)}\left( 1-\frac{(g(a)-\varepsilon)}{2}r^2 \right)\geq \frac{1}{\pi r^2}\int_{\Delta(a;r)}|f|^2e^{-\varphi},
	$$
	which implies that 
	$$
	e^{-\frac{(g(a)-\varepsilon)}{2}r^2}\geq \frac{\int_{\Delta(a;r)}|f|^2e^{-\varphi}}{\pi r^2 e^{-\varphi(a)}}\geq L_\varphi(a,r).
	$$
\end{proof}




The proof of Corollary \ref{cor:douchi2} is the same. 
Note that we do not need to use the Ohsawa--Takegoshi $L^2$-extension theorem in order to prove Corollary \ref{cor:douchi2}, which is one interesting point. 




%
Next, we discuss Pr\'ekopa-type theorems.
As an immediate consequence of Corollary \ref{cor:douchi}, we obtain the following result.

\begin{theorem}\label{thm:prekopateiryouka}
	Let $\varphi:U_\tau\times \Omega_z\to \R$ be a function, which is smooth on $\overline{U\times \Omega}$.
	Here $U_\tau\subset\C_\tau$ and $\Omega_z\subset \C^n_z$ are bounded domains. We define the function $\Phi$ on U by 
	$$
	e^{-\Phi(\tau)}:=\int_{z\in\Omega}e^{-\varphi(\tau,z)}.
	$$
	Let $g$ be a non-negative function on $U$ and $\omega=\ai \dtau\wedge \dbartau$. 
	Then the following properties are equivalent$\colon$ 
	\begin{enumerate}
		\item $\ai \ddbar \Phi\geq g\omega$.
		\item For any point $a\in U$ and any $\varepsilon>0$, there exists $r_\varepsilon\in (0,\delta(a))$ such that 
		$$
		L_\Phi(a, r)\leq e^{-\max\{ \frac{(g(a)-\varepsilon)}{2}, 0\}r^2}
		$$
		for $r\in (0, r_\varepsilon)$. 
	\end{enumerate}
\end{theorem}

This theorem can be seen as a quantification of Pr\'ekopa-type theorems. 
Let us explain the reason for that below.

First, we introduce the following fundamental result obtained by Berndtsson. 

\begin{theorem}[\cite{Ber98}]\label{thm:BerndtssonPrekopa}
	Let $\varphi$ be a smooth psh function on $\overline{U\times \Omega}$, where $U_\tau\subset \C$ is a domain and $\Omega_z\subset \C^n$ is a bounded pseudoconvex domain. 
	Assume that $\Omega$ is a Reinhardt domain and $\varphi$ is independent of $\arg(z_j)$ for $j=1, \ldots, n$. 
	Then $\Phi$ defined by 
	$$
	e^{-\Phi(\tau)}=\int_{z\in\Omega}e^{-\varphi(\tau,z)}
	$$
	is psh as well. 
\end{theorem}

Pr\'ekopa's theorem is originally a result in convex geometry.
The statement is that if $\varphi\colon\R_t\times\R^n_x\to \R$ is a convex function, then the function $\Phi\colon\R\to \R$ defined by $e^{-\Phi(t)}:=\int_{x\in\R^n}e^{-\varphi(t,x)}$ is also convex. 
We call the above type theorem {\it Pr\'ekopa-type theorem} in this article. 
We remark that the smoothness of $\varphi$ and the boundedness of $\Omega$ are not necessary due to an approximate argument, but we go on this setting for simplicity without loss of generality. 

Actually, we can see that Berndtsson's theorem above can be reduced to the equivalence of the plurisubharmonicity of the weight and the optimality of the Ohsawa--Takegoshi $L^2$-extension theorem. 
Here we explain the reason. 
Fix $a\in U$ and $r\in (0, \delta(a))$. 
If $\varphi$ is psh, thanks to the usual optimal $L^2$-extension theorem \cite{Blo13}, \cite{GZ15}, we get a holomorphic function $f(\tau,z)$ on $\Delta(a;r)\times \Omega$ such that $f(a, z)\equiv 1$ and 
\begin{equation*}
\int_{\Delta(a;r)\times \Omega}|f(\tau,z)|^2e^{-\varphi(\tau,z)}\leq \pi r^2 \int_\Omega e^{-\varphi(a, z)}.
\end{equation*}
We take $f$ which has the minimum $L^2$-norm.
This minimum extension is uniquely determined due to the following reason. 
Let $A^2(\Delta(a;r)\times \Omega, \varphi)
=:A$ and $H:=\{ f\in A\mid f|_{\{a\}\times \Omega}\equiv 0\}$. 
We write $A=H\oplus H^\bot$. Then for above $f$, $\pi_{H^\bot}(f)$ is the unique extension with minimum $L^2$-norm.
We simply write $f$ for $\pi_{H^\bot}(f)$. 
Following the argument in \cite{Ber98}, we can say that the minimum extension $f$ has the same invariance property as of $\varphi$, that is, $f$ is independent of $\arg(z_j)$ for $j=1, \ldots, n$. 
Since $f$ is holomorphic in $z_j$, $f$ is also independent of $z_j$, which implies that $f$ is just a function of $\tau$. 
We also denote it by $f$. 
Then 
\[
	\int_{\Delta(a;r)}|f(\tau)|^2e^{-\Phi(\tau)} \leq \pi r^2 e^{-\Phi(a)},
\]
that is, $\Phi$ satisfies the optimal $L^2$-extension property $L_\Phi(a,r)\leq 1$. 
Hence, $\Phi$ is psh. 
In this sense, Theorem \ref{thm:prekopateiryouka} indicates a direction of quantification for Pr\'ekopa's theorem.
Indeed, if sharper estimates can be obtained for each fiber, the strict positivity of $\Phi$ can also be determined accordingly. Conversely, if $\Phi$ has strict positivity, local sharper estimates can be obtained for each fiber.


As another concrete example, sharper estimates can be systematically constructed as follows.


\begin{theorem}\label{thm:prekopatype}
	Let $\varphi:U_\tau\times \Omega_z\to \R$ be a psh function, which is smooth on $\overline{U\times \Omega}$.
	Here $U_\tau\subset\C_\tau$ and $\Omega_z\subset \C^n_z$ are bounded domains. 
	We also assume that $\varphi$ satisfies  
	$$
	\frac{\partial^2\varphi}{\partial\tau\partial\bar{\tau}} (\tau,z)-\left|\frac{\partial\varphi}{\partial\tau}(\tau,z)\right|^2\geq g(\tau),
	$$
	where $g$ is a non-negative function on $U$.
	Let $\Phi:U\to \R$ be the function defined by 
	$$
	e^{-\Phi(\tau)}:=\int_{z\in\Omega}e^{-\varphi(\tau,z)}
	$$
	and $G:U\to \R$ be the function defined by
	$$
	G(\tau)=g(\tau)+\frac{\left| \int_{z\in \Omega}\frac{\partial\varphi}{\partial\tau}(\tau,z)e^{-\varphi(\tau,z)}\right|^2}{\left( \int_{z\in\Omega} e^{-\varphi(\tau,z)}\right)^2}. 
	$$
	Then
	for any point $a\in U$ and any $\varepsilon>0$, there exists $r_\varepsilon\in (0,\delta(a))$ such that for $r\in (0,r_\varepsilon)$
	\begin{equation*}\label{eq:hyouka}
		 L_\Phi(a,r) \leq e^{-\max\{\frac{(G(a)-\varepsilon)}{2},0\}r^2},
	\end{equation*}
that is, there exists a holomorphic function $f_r$ on $\Delta(a;r)$ satisfying $f_r(a)=1$ and 
$$
\int_{\Delta(a;r)\times \Omega} |f_r(\tau)|^2e^{-\varphi(\tau,z)}\leq e^{-\max\{\frac{(G(a)-\varepsilon)}{2},0\}r^2}\pi r^2 \int_{\Omega}e^{-\varphi(a,z)}.
$$
\end{theorem}

\begin{proof}[Proof of Theorem \ref{thm:prekopatype}]
	Since 
	$$
	\Phi(\tau)=-\log \int_{\Omega}e^{-\varphi(\tau,z)}, 
	$$
	we can see that 
	\begin{align*}
		\frac{\partial^2\Phi}{\partial\tau\partial\bar{\tau}}&=\frac{\int_\Omega \left( \frac{\partial\varphi}{\partial\tau\partial\bar{\tau}}- \left| \frac{\partial\varphi}{\partial\tau}\right|^2\right)e^{-\varphi}}{\int_\Omega e^{-\varphi}}+\frac{\left| \int_{\Omega}\frac{\partial\varphi}{\partial\tau}e^{-\varphi}\right|^2}{\left( \int_\Omega e^{-\varphi}\right)^2}\\
		&\geq g+\frac{\left| \int_{\Omega}\frac{\partial\varphi}{\partial\tau}e^{-\varphi}\right|^2}{\left( \int_\Omega e^{-\varphi}\right)^2}=G
	\end{align*}
by simple computation. Note that $G\geq g$ and $G$ is a non-negative function as well.
As a corollary of Theorem \ref{thm:prekopateiryouka}, we can say that for any $a\in U$ and any $\varepsilon>0$, there exists $r_\varepsilon\in(0,\delta(a))$ such that 
$$
L_\Phi(a,r) \leq e^{-\max\{\frac{(G(a)-\varepsilon)}{2},0\}r^2}
$$
for $r\in(0,r_\varepsilon)$. 
\end{proof}

\begin{remark}
	If we only assume that 
	$$
	\frac{\partial^2\varphi}{\partial\tau\partial\bar{\tau}}\geq g
	$$
	in Theorem \ref{thm:prekopatype}, 
	the plurisubharmonicity of $\Phi$ does not follow. 
	This type of example was found by Kiselman \cite{Kis78}.
	Let $\varphi(\tau,z):=|\tau-\bar{z}|^2-|\tau|^2$ be a function on $\C\times \C$. Since $\varphi(\tau,z)=|z|^2-2\Rea(\tau z)$, $\varphi$ is psh and satisfies 
	$$
	\frac{\partial^2\varphi}{\partial\tau\partial\bar{\tau}}\geq 0.
	$$
	Then $\Phi$ defined by 
	$$
	e^{-\Phi(\tau)}=\int_{z\in\C}e^{-\varphi(\tau,z)}
	$$
	is equal to $\Phi(\tau)=-|\tau|^2-C$ for some constant $C$, which is not clearly psh and does not satisfy $L_\Phi\leq 1$. 
	If we want to obtain the plurisubharmonicity of $\Phi$ or the sharper estimate, we need to assume that, for instance, $\ai\ddbar\varphi$ is ``sufficiently" positive such as in Theorem \ref{thm:prekopatype} or  $\varphi$ is invariant under some group actions such as in Theorem \ref{thm:BerndtssonPrekopa} (for further studies, cf. \cite{DZZ14}).
\end{remark}

\section{$L^2$-extension indices for vector bundles and curvature positivity}\label{sec:koujigen}

In this section, we establish a higher rank analogue of the results in Section \ref{sec:1jigen}.
First, we give a proof of Theorem \ref{mainthm:koujigen}.
The proof is almost the same as the proof of Theorem \ref{mainthm:1jigen}. 
One part of the following proof is inspired by the argument in \cite[Theorem 6.4]{DWZZ18}.

\begin{proof}[Proof of Theorem \ref{mainthm:koujigen}]
	Take $a\in\Omega$, $\gamma_a$ and $g_a$ as in Theorem \ref{mainthm:koujigen}. 
	If $g(a)>0$, we take an arbitrary $\varepsilon\in(0,g_a(0))$.
	We also take $r_\varepsilon\in(0,\gamma_a)$ such that $g_a(0)-\varepsilon<g_a(r)$ for $r\in(0, r_\varepsilon)$.
	
	We define $\widetilde{h}:=he^{2(g_a(0)-\varepsilon)|z-a|^2}$.
	Then its dual metric is $\widetilde{h}^\star=h^\star e^{-2(g_a(0)-\varepsilon)|z-a|^2}$.
	We are going to prove that for any local holomorphic section $u$ of $E^\star$ on an open neighborhood of $a$, $|u|^2_{\widetilde{h}^\star}$ is psh. 
	We may assume that $u(a)\neq 0$. 
	By taking $r_\varepsilon$ small enough if necessary, we can assume that $u$ is defined on $\Delta(a;r_\varepsilon)$ and $u\neq 0$ there.

	By definition, there exists $\xi\in E_a$ such that $|\xi|_{\widetilde{h}(a)}=|\xi|_{h(a)}=1$ and $|u(a)|_{\widetilde{h}^\star(a)}=|\langle u(a), \xi\rangle|$.
	From Lemma \ref{lem:vectorattain}, for each $r\in(0, r_\varepsilon)$, there exists $s_r\in A^2(\Delta(a;r),\varphi)$ satisfying $s_r(a)=\xi$ and 
	$$
	L_h(a,r, \xi)=\frac{\int_{\Delta(a;r)}|s_r|^2_h}{\pi r^2 |\xi|^2_{h(a)}}=\frac{\int_{\Delta(a;r)}|s_r|^2_h}{\pi r^2}. 
	$$
	Since $\log |\langle u, s_r\rangle|$ is psh, thanks to the optimal $L^2$-extension theorem \cite{Blo13}, \cite{GZ15}, there exists a holomorphic function $f_r$ on $\Delta(a;r)$ satisfying $f_r(a)=1$ and 
	$$
	\int_{\Delta(a;r)} |f_r|^2e^{-\log |\langle u, s_r\rangle|}\leq \pi r^2 e^{-\log |\langle u(a), \xi\rangle|}=\pi r^2e^{-\log |u(a)|_{\widetilde{h}^\star(a)}}.  
	$$
	We have 
	$
	|u|_{\widetilde{h}^\star}\geq {|\langle u, s_r\rangle|}/{|s_r|_{\widetilde{h}}}, 
	$
	which implies that 
	$$
	e^{-\frac{1}{2}\log |u|_{\widetilde{h}^\star}}\leq |s_r|^{1/2}_{\widetilde{h}}e^{-\frac{1}{2}\log |\langle u, s_r\rangle|}.
	$$
	Then it holds that 
	\begin{align*}
		\int_{\Delta(a;r)}|f_r| e^{-\frac{1}{2}\log |u|_{\widetilde{h}^\star}}&\leq \int_{\Delta(a;r)}|f_r| |s_r|^{1/2}_{\widetilde{h}} e^{-\frac{1}{2}\log |\langle u,s_r\rangle|}\\
		&\leq \left( \int_{\Delta(a;r)}|s_r|_{\widetilde{h}}\right)^{1/2}\left( \int_{\Delta(a;r)} |f_r|^2e^{-\log |\langle u,s_r\rangle|} \right)^{1/2}\\		
		&\leq \left( \int_{\Delta(a;r)}|s_r|_he^{(g_a(0)-\varepsilon)|z-a|^2}\right)^{1/2}\sqrt{\pi}re^{-\frac{1}{2}\log |u(a)|_{\widetilde{h}^\star(a)}}\\
		&\leq \left( \int_{\Delta(a;r)}|s_r|^2_h\right)^{1/4}\left( \int_{\Delta(a;r)}e^{2(g_a(0)-\varepsilon)|z-a|^2} \right)^{1/4}\sqrt{\pi}re^{-\frac{1}{2}\log |u(a)|_{\widetilde{h}^\star(a)}}\\
		&\leq \sqrt[4]{\pi}\sqrt{r}\sqrt{|\xi|_{h(a)}}e^{-\frac{1}{4}g_a(r)r^2}{\sqrt[4]{\pi}}\left( \frac{e^{2(g_a(0)-\varepsilon)r^2}-1}{2(g_a(0)-\varepsilon)}\right)^{1/4}\sqrt{\pi}re^{-\frac{1}{2}\log |u(a)|_{\widetilde{h}^\star(a)}}\\
		&=\pi r^2e^{-\frac{1}{4}g_a(r)r^2}\left( \frac{e^{2(g_a(0)-\varepsilon)r^2}-1}{2(g_a(0)-\varepsilon)r^2}\right)^{1/4}e^{-\frac{1}{2}\log |u(a)|_{\widetilde{h}^\star(a)}}.
	\end{align*}
Here we use the assumption that $L_h(a,r,\xi)\leq e^{-g_a(r)r^2}$. 
By repeating the same argument as in the proof of Theorem \ref{mainthm:1jigen}, for $\delta>0$, there exists $r_\delta>0$ such that 
$$
\left( \frac{e^{2(g_a(0)-\varepsilon)r^2}-1}{2(g_a(0)-\varepsilon)r^2}\right)^{1/4}\leq e^{(\frac{1}{4}+\frac{\delta}{2})(g_a(0)-\varepsilon)r^2} 
$$
for $r\in (0, r_\delta)$. 
Letting $r_{\varepsilon,\delta}:=\min\{ r_\varepsilon, r_\delta\}$, we have that 
\begin{align*}
	\int_{\Delta(a;r)}|f_r| e^{-\frac{1}{2}\log |u|_{\widetilde{h}^\star}}&\leq \pi r^2e^{-\frac{1}{4}g_a(r)r^2}\left( \frac{e^{2(g_a(0)-\varepsilon)r^2}-1}{2(g_a(0)-\varepsilon)r^2}\right)^{1/4}e^{-\frac{1}{2}\log |u(a)|_{\widetilde{h}^\star(a)}}\\
	&\leq \pi r^2e^{-\frac{1}{4}g_a(r)r^2}e^{(\frac{1}{4}+\frac{\delta}{2})(g_a(0)-\varepsilon)r^2} e^{-\frac{1}{2}\log |u(a)|_{\widetilde{h}^\star(a)}}\\
	&\leq \pi r^2 e^{\frac{r^2}{4}(g_a(0)-\varepsilon-g_a(r))}e^{\frac{\delta}{2}g_a(0)r^2}e^{-\frac{1}{2}\log |u(a)|_{\widetilde{h}^\star(a)}}\\
	&\leq \pi r^2 e^{-\frac{1}{2}\log |u(a)|_{\widetilde{h}^\star(a)}}e^{\frac{\delta}{2}g_a(0)r^2}.
\end{align*}
Note that $\frac{1}{2}\log |u|_{\widetilde{h}^\star}$ satisfies the same inequality as (\ref{ineq:1})-(\ref{ineq:5}).
Hence, by repeating the same argument as in the proof of Theorem \ref{mainthm:1jigen}, we can conclude that 
$\log |u|_{\widetilde{h}^\star}$ is psh, that is, $\widetilde{h}$ is Griffiths semi-positive. 
Then it holds that 
$$
\ai\Theta_h(a)\geq 2(g_a(0)-\varepsilon)\omega\otimes \Id_E.
$$
Taking $\varepsilon\to+0$, we finish the proof. 

\end{proof}

As an application, we can establish the higher rank analogue of Corollary \ref{cor:douchi2}. 
We also remark that the following corollary can be used to prove Theorem \ref{mainthm:directimage}. 

\begin{corollary}\label{cor:douchivector}
	Let $\Omega$, $E$ and $h$ be the same things as in Theorem \ref{mainthm:koujigen}.
	We also let $g$ be a semi-positive function on $\Omega$.  
	Fix $a\in \Omega$. 
	Then the following are equivalent$\colon$
	\begin{enumerate}
		\item For any $\varepsilon>0$ and any $\xi\in E_a\sasyugo \{0\}$, there exists $r_\varepsilon\in(0,\delta(a))$ such that 
		$$
		L_h(a, r, \xi)\leq e^{-\max\{\frac{(g(a)-\varepsilon)}{2},0\}r^2}
		$$
		for $r\in (0, r_\varepsilon)$. 
		\item $\ai\Theta_h(a)\geq g(a)\omega \otimes \Id_E$.
	\end{enumerate}
\end{corollary}

\begin{proof}[Proof of Corollary \ref{cor:douchivector}]
	The implication $(1)\Longrightarrow (2)$ is just a consequence of Theorem \ref{mainthm:koujigen} since we have that $\ai\Theta_h(a)\geq \max\{ g(a)-\varepsilon, 0\}\omega\otimes\Id_E\geq (g(a)-\varepsilon)\omega\otimes\Id_E\to g(a)\omega\otimes \Id_E$ as $\varepsilon\to 0$. 
	We only show $(2)\Longrightarrow(1)$. 
	If $\varepsilon\geq g(a)$, we have to prove $L_h(a,r,\xi)\leq 1$.
	This follows from the optimal $L^2$-extension theorem for (Nakano) semi-positive vector bundle. 
	Hence, we only need to consider the case that $\varepsilon < g(a)$. 
	Thanks to \cite[Chapter V, Proposition (12.10)]{Dem-agbook}, 
	we can take a holomorphic frame $(e_1, \ldots,e_r)$ of $E$ such that 
	$$
	h(e_\lambda(z), e_\mu(z))=\delta_{\lambda\mu}-c_{\lambda\mu}|z-a|^2+O(|z-a|^3),
	$$
	where $(\ai\Theta_{h})_a=\sum_{1\leq\lambda,\mu\leq r}c_{\lambda\mu}\dz\wedge\dbarz\otimes e_\lambda^\star\otimes e_\mu$. 
	We write $h_{\lambda\mu}$ for $h(e_\lambda, e_\mu)$ and $\xi=\xi^1e_1(a)+\cdots+\xi^re_r(a)$.
	We define a section $s\in H^0(\Omega, E)$ with constant coefficients by 
	$s(z):=\xi^1e_1(z)+\cdots +\xi^re_r(z)$. Then 
	$$
	|s|^2_h(z)=\sum_{1\leq \lambda,\mu\leq r} h_{\lambda\mu}\xi_\lambda\bar{\xi}_\mu=\sum_{1\leq \lambda\leq r}|\xi_\lambda|^2-|z-a|^2\sum_{1\leq \lambda,\mu\leq r} c_{\lambda\mu}\xi_\lambda\bar{\xi}_\mu+O(|z-a|^3). 
	$$
	Define $\Psi$ by $\Psi=-|s|^2_h$. 
	Then we have 
	$$
	\ai\ddbar\Psi(a)=\langle \ai\Theta_{h}\xi,\xi \rangle_h(a)\geq g(a)|\xi|^2_{h(a)}\omega
	$$
	from the assumption of $(2)$. 
	Set $\widetilde{\Psi}:=\Psi-(g(a)-\varepsilon)|\xi|^2_{h(a)}|z-a|^2$. 
	We then get 
	$$
	\ai\ddbar\widetilde{\Psi}(a)\geq \varepsilon|\xi|^2_{h(a)}\omega.
	$$
	Hence, there exists $r_\varepsilon>0$ such that $\widetilde{\Psi}$ is psh on $\Delta(a;r_\varepsilon)$. 
	It implies that 
	$$
	\frac{1}{\pi r^2}\int_{\Delta(a;r)}\widetilde{\Psi}\geq \widetilde{\Psi}(a)=-|s|^2_h(a)=-|\xi|^2_{h(a)}
	$$
	for $r\in (0,r_\varepsilon)$. 
	Repeating the same argument as in the proof of Corollary \ref{cor:douchi}, we can get 
	$$
	e^{-\frac{(g(a)-\varepsilon)}{2}r^2}\geq \frac{\int_{\Delta(a;r)}|s|^2_h}{\pi r^2 |\xi|^2_{h(a)}}\geq L_h(a,r,\xi).
	$$
\end{proof}

Considering the case that $\Omega=B$, $E=\pi_\star(K_{\mathcal{X}/B}\otimes \mathcal{L})$ and $h=H$, we can prove Theorem \ref{mainthm:directimage}. 

\begin{proof}[Proof of Theorem \ref{mainthm:directimage}]
	$(1)\Longrightarrow(2)$. Fix $a\in B$, $\xi\in E_a=H^0(X_a,K_{X_a}\otimes L_a)$ and $\varepsilon>0$. 
	We may assume that $\xi\neq 0$. 
	Corollary \ref{cor:douchivector} asserts that there exists $r_\varepsilon\in(0,\delta(a))$ such that for any $r\in(0,r_\varepsilon)$
	$$
	L_{H}(a,r,\xi)\leq e^{-\max\{\frac{(g(a)-\varepsilon)}{2},0\}r^2}. 
	$$
	Thanks to Lemma \ref{lem:vectorattain}, we get a holomorphic section $u\in A^2(\Delta(a;r), H)$ satisfying $u(a)=\xi$ and 
	$$
	\int_{\Delta(a;r)}|u|^2_H\leq e^{-\max\{\frac{(g(a)-\varepsilon)}{2},0\}r^2}\pi r^2 |\xi|^2_{H(a)}.  
	$$
	Define 
	$$
	s:=u\wedge \dtau\in H^0(\Delta(a;r), K_B\otimes E)=H^0(\pi^{-1}(\Delta(a;r), K_\mathcal{X}\otimes \mathcal{L}).
	$$
	Then $s$ satisfies $s|_{X_a}=\xi\wedge \dtau$ and 
	$$
	\int_{\pi^{-1}(\Delta(a;r))}c_{n+1}hs\wedge \bar{s}\leq e^{-\max\{\frac{(g(a)-\varepsilon)}{2},0\}r^2}\pi r^2 \int_{X_a} c_nh_a\xi\wedge\bar{\xi}.
	$$
	
	$(2)\Longrightarrow(1)$. 
	Take $a,\xi,\varepsilon, r_\varepsilon, r,s$ satisfying the condition of $(2)$ in Theorem \ref{mainthm:directimage}. 
	Then $s\in H^0(\pi^{-1}(\Delta(a;r), K_\mathcal{X}\otimes \mathcal{L})$ defines the unique section $u\in H^0(\Delta(a;r), \pi_\star(K_{\mathcal{X}/B}\otimes \mathcal{L}))$ such that $s=u\wedge\dtau$. 
	Then $u$ satisfies 
	$$
	\frac{\int_{\Delta(a;r)}|u|^2_H}{\pi r^2 |\xi|^2_{H(a)}}\leq e^{-\max\{\frac{(g(a)-\varepsilon)}{2},0\}r^2}
	$$
	as well, which implies that $L_H(a,r,\xi)\leq e^{-\max\{\frac{(g(a)-\varepsilon)}{2},0\}r^2}$ for $r\in(0,r_\varepsilon)$. 
	By Corollary \ref{cor:douchivector}, we see that 
	$$
	\ai\Theta_H(a)\geq g(a)\omega\otimes\Id_{\pi_\star(K_{\mathcal{X}/B}\otimes \mathcal{L})}. 
	$$
	
\end{proof}

If $\mathcal{L}$ is a positive line bundle, $\pi_\star(K_{\mathcal{X}/B}\otimes \mathcal{L})$ is positive as well (cf. \cite{Ber09}). 
Hence, in this setting, we get a sharper estimate from each fiber locally. 
Conversely, if we get a sharper estimate from each fiber, we can say that $(\pi_\star(K_{\mathcal{X}/B}\otimes \mathcal{L}), H)$ is strictly positive.

\section{New characterization of harmonicity}\label{sec:harmonicity}
As mentioned in Introduction, it is an interesting attempt to give a new characterization of a certain kind of curvature flatness in terms of the optimal Ohsawa--Takegoshi $L^2$-extension theorem. 
In this section, we firstly give a new characterization of harmonicity and a proof of Theorem \ref{mainthm:harmoniccharacterization}. 

\begin{proof}[Proof of Theorem \ref{mainthm:harmoniccharacterization}]
	(1)$\Longrightarrow$(2). 
We fix any $a\in \Omega$ and $r\in (0,\delta(a))$. 
Using Lemma \ref{lem:attain}, we take a holomorphic function $f$ on $\Delta(a;r)$ satisfying $f(a)=1$ and 
\[
\frac{\int_{\Delta(a;r)}|f|^2e^{-\varphi}}{\pi r^2e^{-\varphi(a)}}=L_\varphi(a,r). 
\]
Since $\varphi$ is harmonic, and subharmonic as well, due to the optimal $L^2$-extension theorem, we have that $L_\varphi(a,r)\leq 1$, that is, 
\[
	\frac{\int_{\Delta(a;r)}|f|^2e^{-\varphi}}{\pi r^2e^{-\varphi(a)}}\leq 1. 
\]
On the other hand, since $\varphi$ is harmonic, there exists a holomorphic function $h$ on $\Delta(a;r)$ such that $2\Rea(h)=\varphi$. 
Hence, for any holomorphic function $g$ on $\Delta(a;r)$, $|g|^2e^{-\varphi}=|ge^{-h}|^2$ is subharmonic. 
Then we get 
\[
\int_{\Delta(a;r)}|f|^2e^{-\varphi} \geq \pi r^2 e^{-\varphi(a)}. 	
\]
Consequently, it holds that 
\[
\int_{\Delta(a;r)}|f|^2e^{-\varphi}	= \pi r^2 e^{-\varphi(a). }
\]
Since the extension $f$ with minimum $L^2$-norm is uniquely determined (cf. the argument in Section \ref{sec:1jigen}), the proof is completed. 

	(2)$\Longrightarrow$(1). 	Theorem \ref{thm:minimalextensionproperty} implies that $\ai\ddbar\varphi\geq 0$. 
	If $\ai\ddbar\varphi(a)>0$ for some point $a\in\Omega$, we take $c>0$ such that $\ai\ddbar\varphi(a)\geq c\omega$. 
	Then there exists $r'>0$ such that for any $r\in(0,r')$ 
	$$
	L_\varphi(a, r)\leq e^{-\frac{c}{4}r^2}
	$$
	from Corollary \ref{cor:douchi2}. 
	This contradicts $L_\varphi\equiv 1$. 
\end{proof}

This theorem implies that if one cannot get any sharper estimates with respect to $\varphi$, $\varphi$ must be harmonic, and vice versa. 
As a higher rank analogue of this theorem, we prove Theorem \ref{mainthm:flatness}. 


\begin{proof}
(1)$\Longrightarrow$(2). 
Take $(a,r,,\xi)\in \Omega_\delta\times_\Omega E\sasyugo \{0\} $. 
By using Lemma \ref{lem:vectorattain}, there exists a holomorphic section $s\in A^2(\Delta(a;r),h)$ such that $s(a)=\xi$ and 
\[
L_h(a,r,\xi)=\frac{\int_{\Delta(a;r)}|s|^2_h}{\pi r^2 |\xi|^2_{h(a)}}. 
\]
Thanks to the optimal $L^2$-extension theorem for (Nakano) semi-positive vector bundle, we see that $L_h(a,r,\xi)\leq 1$, that is, 
\[
	\frac{\int_{\Delta(a;r)}|s|^2_h}{\pi r^2 |\xi|^2_{h(a)}} \leq 1. 
\]
On the other hand, we fix any holomorphic section $u\in A^2(\Delta(a;r), h)$ with $u(a)=\xi$. 
The curvature formula implies that 
\begin{align*}
\ai\ddbar |u|^2_h &= - \langle \ai\Theta_h u,u\rangle_h + \ai \langle D'_h u, D'_h u\rangle_h \\
& = \ai \langle D'_h u, D'_h u\rangle_h \geq 0.
\end{align*}
Therefore, $|u|^2_h$ is a subharmonic function. 
Then it holds that 
\[
\frac{\int_{\Delta(a;r)} |s|^2_h}{\pi r^2 |\xi|^2_{h(a)}} \geq 1. 
\]
Hence, we get 
\[
	\frac{\int_{\Delta(a;r)}|s|^2_h}{\pi r^2 |\xi|^2_{h(a)}} = 1.
\]
The uniqueness of $s$ also follows (cf. the argument in Section \ref{sec:1jigen}).

(2)$\Longrightarrow$(1).
	Thanks to the optimal $L^2$-extension property for a holomorphic vector bundle (cf. \cite{DNWZ22}), we can say that $\ai\Theta_h\geq 0$. 
 If $\ai\Theta_h \not\equiv 0$, there exist $a\in \Omega$ and $\xi\in E_{a}$ with $\xi\neq 0$ such that $\langle \ai\Theta_h \xi, \xi\rangle_{h}(a)>0$. 
 We fix a positive number $c>0$ satisfying $\langle \ai\Theta_h \xi, \xi\rangle_{h}(a) \geq c |\xi|^2_{h(a)}\omega$.
 Similar to the proof of Corollary \ref{cor:douchivector}, we take a holomorphic frame $(e_1,\ldots, e_r)$ of $E$ and define a section $s$ with constant coefficients by $s=\xi^1 e_1+\cdots+\xi^re_r$, where $\xi=\xi^1e_1(a)+\cdots + \xi^re_r(a)$. 
 Taking $\varepsilon=c/2$ and repeating the argument there, we can prove that the $L^2$-extension index in $\xi$-direction satisfies
 \[
 L_h(a,r,\xi)\leq e^{-\frac{c}{4}r^2}
 \]
for any $r\in (0,r_c)$, where $r_c$ is some positive constant. 
 This contradicts $L_h\equiv 1$.
\end{proof}

\appendix

\section{Further studies and related problems}\label{sec:further}
Shortly after this paper appeared on the arXiv, a new paper by Wang Xu was announced \cite{Xu22}, which generalized Theorem \ref{mainthm:koujigen} to higher dimensional cases. It seems interesting to generalize the definition of $L^2$-extension indices to higher dimensions and reformulate his results by using this notion.
In the appendix, we give a definition of $L^2$-extension indices for smooth Hermitian metrics on holomorphic vector bundles over an $n$-dimensional domain. 
We prepare notation. 
Let $\Omega$ be a bounded domain in $\C^n$, $\pi\colon E\to \Omega$ be a holomorphic vector bundle over $\Omega$ and $h$ be a smooth Hermitian metric on $E$. 
We also set $\Delta_r=\{ z\in \C \mid |z|<r\}$, $\mathbb{B}^m_s=\{ z\in \C^m \mid |z|<s \}$ and $P_{A,r,s}=A(\Delta_r\times\mathbb{B}_s^{n-1})$ for $r,s>0$ and $A\in \mathbf{U}(n)$.
Here $\mathbf{U}(n)$ is the set of all $n$-dimensional unitary groups and $P_{A,r,s}$ is nothing but holomorphic cylinder. 
Set
$\Omega_{\widetilde{\delta}}=\{ (a,r,s,A)\in\Omega\times \R_{>0}\times\R_{>0}\times\mathbf{U}(n) \mid a+P_{A,r,s}\subset \Omega\}$ and $\Omega_{\widetilde{\delta}}\times_{\Omega} E\sasyugo\{ 0\}=\{ (a,r,s,A,\xi)\in \Omega_{\widetilde{\delta}}\times E\sasyugo\{ 0\} \mid a=\pi (\xi)\}$. 
We define an {\it $L^2$-extension index} $L_h\colon\Omega_{\widetilde{\delta}}\times_{\Omega} E\sasyugo\{ 0\}\to \R$ by 
\begin{equation}\label{eq:l2indiceshigher}
L_h(a,r,s,A,\xi)=\inf \left\{ \frac{\int_{a+P_{A,r,s}}|s|^2_h}{|P_{A,r,s}||\xi|^2_{h(a)}} ~\middle|~ s\in A^2(a+P_{A,r,s}, h), s(a)=\xi \right\},
\end{equation}
where $|P_{A,r,s}|$ is the volume of $P_{A,r,s}$. 
Here we follow the convention in Section \ref{sec:prelim}. 
Note that this definition is a generalization of Definition \ref{def:l2indexvector}. 

We explain why we take holomorphic cylinders. 
To get the Griffiths positivity of the metric (or more simply the plurisubharmonicity of the weight), we need to consider not polydiscs but holomorphic cylinders in (\ref{eq:l2indiceshigher}). 
See \cite[Lemma 3.1 and Remark 3.2]{DNW21} and \cite[Theorem 1.3]{DNWZ22} for detail. 

The same results discussed in the previous sections
can be established in parallel.
Indeed, for example, we can establish a characterization of pluriharmonic functions, which is a higher rank analogue of Theorem \ref{mainthm:harmoniccharacterization}. 

\begin{theorem}\label{thm:pluriharmonic}
	Let $\varphi$ be a smooth function on a domain $\Omega\subset \C^n$. Then the following are equivalent$\colon$
	\begin{enumerate}
		\item $\ai\ddbar\varphi=0$.
		\item $L_\varphi=1$, that is, for any $(a,r,s,A)\in \Omega_{\widetilde{\delta}}$, there exists a \textit{\textbf{unique}} holomorphic function $f$ on $a+P_{A,r,s}$ satisfying $f(a)=1$ and \[
		\int_{a+P_{A,r,s}}|f|^2e^{-\varphi}\leq |P_{A,r,s}|e^{-\varphi(a)}.	
		\]
	\end{enumerate}
\end{theorem}

If we are only dealing with plurisubharmonicity or Griffiths positivity, it is sufficient to consider only the one-dimensional case.

We can also extend the definition of the $L^2$-extension index to general upper semi-continuous functions. 
For an upper semi-continuous function $\varphi$ on a bounded domain $\Omega \subset \C^n$, 
we define the $L^2$-extension index $L_\varphi$ of $\varphi$ on $\Omega_{\widetilde{\delta}}$ as follows: for $(a,r,s,A)\in \Omega_{\widetilde{\delta}}$, if $\varphi(a)>-\infty$,
\begin{align*}
L_{\varphi}(a,r,s,A):&= \frac{1}{|P_{r,s,A}|K_{P_{r,s,A}, \varphi}(a)}\\
&=\inf \left\{ \frac{\int_{a+P_{r,s,A}}|f|^2e^{-\varphi}}{|P_{r,s,A}|e^{-\varphi(a)}} ~\middle|~ f\in A^2(a+P_{r,s,A}, \varphi) \And f(a)=1 \right\},
\end{align*}
and if $\varphi(a)=-\infty$, $L_{\varphi}(a,r,s,A)=0$.
By using this concept, we can consider various generalizations of the theorems established in this paper.
For example, as a generalization of Theorems \ref{mainthm:harmoniccharacterization} and \ref{thm:pluriharmonic}, we can consider the following conjecture.

\begin{conjecture}\label{conj:pluriharmonic}
Let $\varphi$ be an \textit{\textbf{upper semi-continuous}} function on $\Omega\subset \C^n$. Then the following are equivalent$\colon$
\begin{enumerate}
	\item $\varphi$ is pluriharmonic.
	\item $L_\varphi=1$, that is, $\varphi>-\infty$ and for any $(a,r,s,A)\in \Omega_{\widetilde{\delta}}$, there exists a unique holomorphic function $f$ on $a+P_{A,r,s}$ satisfying $f(a)=1$ and \[
	\int_{a+P_{A,r,s,}}|f|^2e^{-\varphi} \leq |P_{A,r,s}|e^{-\varphi(a)}. 	
	\]
\end{enumerate}
\end{conjecture}	
The formulation (2) in Conjecture \ref{conj:pluriharmonic} does not assume the regularity of the function.
It would be interesting that regularity, such as pluriharmonicity, could be obtained from such a condition. 

\begin{remark}
    After uploading this paper, there have been some developments regarding conjecture \ref{conj:pluriharmonic}. 
    In the case of continuous functions, we proved the conjecture to be correct in \cite{Ina23}. 
    Subsequently, Liu and Xu proved it for general $\R$-valued measurable functions in \cite{LX23}.
\end{remark}








\begin{thebibliography}{99}
 \bibitem[Ber98]{Ber98} B.~Berndtsson, \emph{Prekopa's theorem and Kiselman's minimum principle for plurisubharmonic functions}, Math. Ann. \textbf{312} (1998), 785--792. 
 \bibitem[Ber06]{Ber06} B.~Berndtsson, \emph{Subharmonicity properties of the Bergman kernel and some other functions associated to pseudoconvex domains}, Ann. Inst. Fourier, \textbf{56} (2006), 1633--1662.
 \bibitem[Ber09]{Ber09} B.~Berndtsson, \emph{Curvature of vector bundles associated to holomorphic fibrations}, Ann. of Math. \textbf{169} (2009), no.~2, 531--560. 
 \bibitem[B\l o13]{Blo13} Z.~B\l ocki, \emph{Suita conjecture and the Ohsawa-Takegoshi extension theorem}, Invent. Math. \textbf{193}, (2013), 149--158. 
 \bibitem[Dem-book]{Dem-agbook} J.-P. Demailly, \emph{Complex analytic and differential geometry}, http://www-fourier.ujf-grenoble.fr/~demailly/manuscripts/agbook.pdf.
 \bibitem[DNW21]{DNW21} F.~Deng, J.~Ning, and Z.~Wang, \emph{Characterizations of plurisubharmonic functions}, Sci. China Math. \textbf{64}, (2021), 1959--1970.
 \bibitem[DNWZ23]{DNWZ22} F.~Deng, J.~Ning, Z.~Wang, and X.~Zhou, \emph{Positivity of holomorphic vector bundles in terms of $L^p$-estimates for $\overline{\partial}$}, Math. Ann. \textbf{385} (2023), 575--607. 
 \bibitem[DWZZ18]{DWZZ18} F.~Deng, Z.~Wang, L.~Zhang, and X.~Zhou, \emph{New characterizations of plurisubharmonic functions and positivity of direct image sheaves}, arXiv:1809.10371.
 \bibitem[DZZ14]{DZZ14} F.~Deng, H.~Zhang, and X.~Zhou, \emph{Positivity of direct images of positively curved volume forms}, Math. Z. \textbf{278} (2014), no.~1-2, 347--362. 
 \bibitem[GZ15]{GZ15} Q.~Guan and X.~Zhou, \emph{A solution of an $L^2$ extension problem with an optimal estimate and applications}, Ann. of Math. (2) \textbf{181} (2015), no.~3, 1139--1208.
 \bibitem[HPS18]{HPS18} C.~Hacon, M.~Popa, and C.~Schnell, \emph{Algebraic fiber spaces over Abelian varieties: around a recent theorem by Cao and Paun}, Contemp. Math. \textbf{712} (2018), 143--195. 
 \bibitem[Hos19]{Hos19} G.~Hosono, \emph{On sharper estimates of Ohsawa--Takegoshi $L^2$-extension theorem}, J. Math. Soc. Japan \textbf{71} (2019), no.~3, 909--914.
 \bibitem[HI21]{HI21} G.~Hosono and T.~Inayama, \emph{A converse of H\"ormander's $L^2$-estimate and new positivity notions for vector bundles}, Sci. China Math. \textbf{64} (2021), 1745--1756. 
 \bibitem[Ina22]{Ina22JGA} T.~Inayama, \emph{Optimal $L^2$-extensions on tube domains and a simple proof of Pr\'ekopa's theorem}, J. Geom. Anal. \textbf{32}, 32 (2022). https://doi.org/10.1007/s12220-021-00796-w.
 \bibitem[Ina23]{Ina23} T.~Inayama, \emph{A note on characterizing pluriharmonic functions via the Ohsawa--Takegoshi extension theorem}, J. Math. Sci. Univ. Tokyo 30 (2023), no.~3, 365--369. 
 \bibitem[KP21]{KP21} A.~Khare and V.~Pingali, \emph{On an asymptotic characterisation of Griffiths semipositivity}, Bull. des Sci. Math. \textbf{167} (2021), 102956.
 \bibitem[Kik23]{Kik22} S.~Kikuchi, \emph{On sharper estimates of Ohsawa--Takegoshi $L^2$-extension theorem in higher dimensional case}, manuscripta math. \textbf{170} (2023), 453--469.
 \bibitem[Kis78]{Kis78} C.~Kiselman, \emph{The partial Legendre transformation for plurisubharmonic functions}, Invent. Math. \textbf{49} (1978), no.~2, 137--148. 
 \bibitem[LX23]{LX23} Z.~Liu and W.~Xu, \emph{Characterizations of Griffiths positivity, pluriharmonicity and flatness}, arXiv:2210.17361v3. 
 \bibitem[OT87]{OT87} T.~Ohsawa and K.~Takegoshi, \emph{On the extension of $L^2$ holomorphic functions}, Math. Z. \textbf{195}, (1987), no.~2, 197--204. 
\bibitem[Siu98]{Siu98} Y.~T.~Siu, \emph{Invariance of plurigenera}, Invent. Math. \textbf{134} (1998), 661--673.
\bibitem[Xu22]{Xu22} W.~Xu, \emph{A Quantitative Characterization of Griffiths Positivity}, arXiv:2210.17361v1. 
\bibitem[XZ22]{XZ22} W.~Xu and Z.~Zhou, \emph{Optimal $L^2$ Extensions of Openness Type}, arXiv:2202.04791.
\end{thebibliography}
\end{document}